\documentclass[runningheads, envcountsame, a4paper]{llncs}
\usepackage{latexsym}
\usepackage{amssymb}
\usepackage{tikz}
\usepackage{graphicx}
\usepackage{amsmath}
\usepackage{bbm}
\usepackage{mathrsfs}
\usepackage{cite}
\usepackage{wrapfig}
\usepackage{verbatim}

\newcommand{\versie}[2]{#1}

\newcommand{\fib}{\mbox{\sf fib}}
\newcommand{\Omicron}{O}
\newcommand{\bcdot}{$\discretionary{\mbox{$ \cdot $}}{}{}$}

\begin{document}

\title{DFAs and PFAs with Long Shortest Synchronizing Word Length}
\titlerunning{DFAs and PFAs with Long Shortest Synchronizing Word Length}
\toctitle{DFAs and PFAs with Long Shortest Synchronizing Word Length}
\author{Michiel de Bondt\inst{1} \and Henk Don\inst{2} \and Hans Zantema\inst{3,1}}
\authorrunning{M. de Bondt, H. Don, and H. Zantema}
\tocauthor{M.~de~Bondt, H.~Don, and H.~Zantema}
\institute{
Department of Computer Science, Radboud University Nijmegen, The Netherlands, \\
email: {\tt m.debondt@math.ru.nl}
\and
Department of Mathematics, Free University, Amsterdam, The Netherlands, \\
email: {\tt h.don@vu.nl}
\and
Department of Computer Science, TU Eindhoven, The Netherlands, \\
email: {\tt h.zantema@tue.nl}
}

\maketitle
\setcounter{footnote}{0}

\begin{abstract}
It was conjectured by \v{C}ern\'y in 1964, that a synchronizing DFA on $n$ states always has a synchronizing word of length at most $(n-1)^2$, and he gave a sequence of DFAs for which
this bound is reached. Until now a full analysis of all DFAs reaching this bound was only given for $n
\leq 4$, and with bounds on the number of symbols for $n \leq 10$. Here we give the full analysis for $n
\leq 6$, without bounds on the number of symbols.

For PFAs on $n\leq 6$ states we do a similar analysis as for DFAs and
find the maximal shortest synchronizing word lengths, exceeding $(n-1)^2$ for $n =4,5,6$. For arbitrary $n$
we use rewrite systems to construct a PFA on three symbols with exponential shortest synchronizing word
length, giving significantly better bounds than earlier exponential constructions. We give a transformation
of this PFA to a PFA on two symbols keeping exponential shortest synchronizing word length, yielding a
better bound than applying a similar known transformation.

\end{abstract}

\section{Introduction and Preliminaries}

A {\em deterministic finite automaton (DFA)} over a finite alphabet $\Sigma$ is called
{\em synchronizing}, if it admits a {\em synchronizing word}. A word $w \in \Sigma^*$ is
called {\em synchronizing} (or directed, or reset), if, starting in any state $q$, after
reading $w$, one always ends in one particular state $q_s$. So reading $w$ acts as a reset
button:
no matter in which state the system is, it always moves to the particular state $q_s$.
Now \v{C}ern\'y's conjecture \cite{C64} states:
\begin{quote}
Every synchronizing DFA on $n$ states admits a synchronizing word of length $\leq (n-1)^2$.
\end{quote}

Surprisingly, despite extensive effort, this conjecture is still open, and even the best known upper
bounds are still cubic in $n$. In 1983 Pin \cite{pin} established the bound $\frac{1}{6}(n^3-n)$\versie{, based on \cite{frankl}}{}. Only very recently a slight improvement was claimed by Szyku\l{}a \cite{szykula}. For a survey on synchronizing automata and \v{C}ern\'y's conjecture, we refer to \cite{volkov}.

Formally, a {\em deterministic finite automaton (DFA)} over a finite alphabet $\Sigma$ consists of a finite set
$Q$ of states and a map $\delta: Q \times \Sigma \to Q$.\footnote{For synchronization the
initial state and the set of final states in the standard definition may be ignored.}
For $w \in \Sigma^*$ and $q \in Q$, we define $qw$ inductively by
$q \lambda = q$ and $q w a = \delta(qw,a)$ for
$a \in \Sigma$, where $\lambda$ is the empty word.
So $qw$ is the state where one ends, when starting in $q$ and reading the symbols in $w$ consecutively, and $qa$ is a short hand notation for
$\delta(q,a)$. A word $w \in \Sigma^*$ is called {\em synchronizing}, if a
state $q_s \in Q$ exists such that $q w = q_s$ for all $q \in Q$. 

In \cite{C64}, \v{C}ern\'y already gave DFAs for which the bound of the conjecture is attained:
for $n \geq 2$ the DFA $C_n$ is defined
to consist of $n$ states $1,2,\ldots,n$, and two symbols $a,b$, acting by
$qa = q+1$ for $q = 1,\ldots,n-1$, $\delta(n,a) = 1$, and  $qb = q$ for
$q = 2,\ldots,n$, $1b = 2$.

\begin{wrapfigure}{r}{4cm}
\begin{tikzpicture}
\useasboundingbox (-1,-0.5) rectangle (3,1.8);
\node[circle,draw,inner sep=0pt,minimum width=5mm] (1) at (0,2) {$1$};
\node[circle,draw,inner sep=0pt,minimum width=5mm] (2) at (2,2) {$2$};
\node[circle,draw,inner sep=0pt,minimum width=5mm] (3) at (2,0) {$3$};
\node[circle,draw,inner sep=0pt,minimum width=5mm] (4) at (0,0) {$4$};
\draw[->] (1) -- node[above,inner sep=1pt] {$a,b$} (2);
\draw[->] (2) -- node[right,inner sep=2pt] {$a$} (3);
\draw[->] (3) -- node[below,inner sep=3pt] {$a$} (4);
\draw[->] (4) -- node[left,inner sep=2pt] {$a$} (1);
\draw[->] (2) edge[out=90,in=0,looseness=7] node[right,inner sep=4pt] {$b$} (2);
\draw[->] (3) edge[out=-90,in=0,looseness=7] node[right,inner sep=4pt] {$b$} (3);
\draw[->] (4) edge[out=-90,in=-180,looseness=7] node[left,inner sep=4pt] {$b$} (4);
\node at (1,-0.9) {$C_4$};
\end{tikzpicture}
\end{wrapfigure}

For $n=4$, this is depicted on the right.
For $C_n$, the string $w = b (a^{n-1}b)^{n-2}$ of length $|w| = (n-1)^2$ satisfies
$qw = 2$ for all $q
\in Q$, so $w$ is synchronizing. No shorter synchronizing word exists for
$C_n$, as is shown in \cite{C64}, showing that the bound in \v{C}ern\'y's conjecture is sharp.

A DFA on $n$ states is {\em critical}, if its shortest synchronizing word has length $(n-1)^2$. 
One goal of this paper is to investigate all critical DFAs up to some size.
To exclude infinitely many trivial extensions, we only consider {\em basic} DFAs: no two distinct
symbols act in the same way in the automaton, and no symbol acts as the identity. Obviously, adding
the identity or copies of existing symbols has no influence on synchronization.

An extensive investigation was already done by Trahtman in \cite{T06}:
by computer support and clever algorithms,
all critical DFAs on $n$ states and $k$ symbols were
investigated for $3 \leq n \leq 7$ and $k \leq 4$, and for $n=8,9,10$ and $k=2$. Here, a
minimality requirement was added: examples were excluded if criticality may be kept after
removing one symbol. Then up to isomorphism there are exactly 8 of them, apart
from the basic \v{C}ern\'y examples: 3 with 3 states, 3 with 4, one with 5 and one with
6. In \cite{DZ16}, the minimality requirement and restrictions on alphabet size were dropped and several more examples were found
that Trahtman originally expected not to exist.
All these are extensions of known examples: in total there are exactly 15 basic critical DFAs for $n=3$ and
exactly 12 basic critical DFAs for $n=4$. In this paper, we show that for $n=5,6$, no more critical DFAs exist
than the four known ones, without any restriction on the number of symbols.

A generalization of a DFA is a {\em Partial Finite Automaton} (PFA); the only difference is that now the
transition function $\delta$ is allowed to be partial. In a PFA, $qw$ may be undefined, in fact it is only defined if every step is
defined. A word $w \in \Sigma^*$ is called {\em carefully synchronizing} for a PFA, if a
state $q_s \in Q$ exists such that $q w$ is defined and  $q w = q_s$ for all $q \in Q$. Stated in words:
starting in any state $q$ and reading $w$, every step is defined and one always ends in state $q_s$.
As being a generalization of DFAs, the shortest carefully synchronizing word may be longer. For $n=4,5,6$
we show that this is indeed the case by finding the maximal shortest carefully synchronizing word length
to be 10, 21 and 37, respectively. The maximal length grows exponentially in $n$, as was already observed by Rystsov \cite{rystsov}. Martyugin \cite{M10} established the lower bound $\Omega(3^{n/3})$ with a construction in which the number of symbols is linear in $n$. In a recent paper, the upper bound $O((3+\varepsilon)^\frac{n}{3})$ was proved \cite{gusev}.

Until recently it was an open question if exponential lower bounds can be achieved with a constant alphabet size. We answer this question by giving a
construction of a PFA on $n$ states and three symbols with exponential shortest synchronizing word length. The key idea is that synchronization is forced to mimic exponentially many string
rewrite steps, similar to binary counting. Our three-symbol PFA can be transformed to a two-symbol PFA by a
standard construction for which we develop a substantial improvement.
Independent of our work, recently in \cite{V16} it was shown that exponential bounds exist for every constant alphabet
size and for two symbols the bound $\Omega(2^{n/35})$ was given. Our basic construction strongly improves this and
gives length $\Omega(\phi^{n/3})$ for the three-symbol PFA and length $\Omega(\phi^{n/5})$ for the
two-symbol PFA, where $\phi = \frac{1+ \sqrt{5}}{2}$. Some optimizations yield further
improvements. 

The basic tool to analyze (careful) synchronization is the {\em power automaton}.
For any DFA or PFA $(Q,\Sigma, \delta)$, its power automaton is the DFA $(2^Q,\Sigma, \delta')$
where $\delta' : 2^Q \times \Sigma \to 2^Q$ is defined by
$\delta'(V,a) = \{q \in Q \mid \exists p \in V : \delta(p,a) = q \}$, if $\delta(p,a)$ is defined for
all $p \in V$, otherwise $\delta'(V,a) = \emptyset$.
For any $V \subseteq Q, w \in \Sigma^*$, we define $Vw$ as above, using $\delta'$ instead of
$\delta$.  From this definition, one easily proves that $Vw = \{ qw \mid q \in V \}$ if $qw$ is defined for 
all $q \in V$, otherwise $Vw = \emptyset$, for any $V \subseteq Q, w \in \Sigma^*$.
A set of the shape $\{q\}$ for $q \in Q$ is called a {\em singleton}.
So a word $w$ is (carefully) synchronizing, if and only if $Qw$ is a singleton.
Hence a DFA (PFA) is (carefully) synchronizing, if and only if its
power automaton admits a path from $Q$ to a singleton, and the shortest
length of such a path corresponds to the shortest length of a (carefully) synchronizing word.

This paper is organized as follows. In Section \ref{secdfa} we describe our exhaustive analysis of
DFAs on at most 6 states. In Section \ref{secnfa} we give our results for PFAs on at most 6 states.
In Section \ref{secexp} we present our construction of PFAs on three symbols with
exponential shortest carefully synchronizing word length.
In Section \ref{sectwos} we improve the transformation used by Martyugin \cite{M10} and Vorel \cite{V16} to reduce to alphabet size two. Section \ref{secopt} discusses optimizations. We conclude in
Section \ref{secconcl}.

\section{Critical DFAs on at Most 6 States}
\label{secdfa}

A natural question when studying \v{C}ern\'y's conjecture is: what can be said about automata in which the
bound of the conjecture is actually attained, the so-called critical automata?
Throughout this section we restrict ourselves to basic DFAs. As has already been noted by several authors
\cite{volkov,T06,DZ16}, critical DFAs are rare. There is only one construction known which gives a critical
DFA for each $n$, namely the well-known sequence $C_n$, discovered by and named after \v{C}ern\'y \cite{C64}.
Apart from this sequence, all known critical DFAs have at most 6 states. In \cite{DZ16}, all critical DFAs
on less than 5 states were identified, without restriction on the size of the alphabet. For $n=5$ and 6
it was still an open question if there exist critical (or even supercritical) DFAs,
other than those already discovered by \v{C}ern\'y, Roman \cite{R08} and Kari \cite{K01}. In this paper we
verify that this is not the case, so for $n=5$ only two critical DFAs exist (\v{C}ern\'y, Roman) and also for
$n=6$ only two exist (\v{C}ern\'y, Kari). In fact our results also prove the following theorem (previously only known for $n\leq 5$, see \cite{CPR71}):

\begin{theorem}
Every synchronizing DFA on $n\leq 6$ states admits a synchronizing word of length at most $(n-1)^2$.
\end{theorem}

\versie{As Trahtman already noted in his paper \cite{T06}, for $n\geq 6$ there seems to be a gap in the range of
possible shortest synchronization lengths. For example, his analysis showed that there are no DFAs on 6
states with shortest synchronizing word length 24, when restricting to at most 4 symbols. Our analysis
shows that this is true without restriction on the alphabet: there is no DFA on 6 states with shortest
synchronizing word length 24. For $n\leq 6$ all other lengths are feasible. If $n\leq 6$ and
$1\leq k\leq (n-1)^2$, $k\neq 24$, then there exists a DFA on $n$ states with shortest synchronizing word
length exactly $k$.}{}

As the number of DFAs on $n$ states grows like $2^{n^{n}}$, an exhaustive search is a non-trivial affair,
even for small values of $n$. The problem is that the alphabet size in a basic DFA can be as large as
$n^n-1$. Up to now for $n=5,6$ only DFAs with at most four symbols were checked by Trahtman \cite{T06}.
Here we give describe our algorithm to investigate all DFAs on 5 and 6 states, without restriction on
the alphabet size.

Before explaining the algorithm, we introduce some terminology. A DFA $\mathcal{B}$ obtained by adding
some symbols to a DFA $\mathcal{A}$ will be called an \emph{extension} of $\mathcal{A}$. If $\mathcal{A} =
(Q,\Sigma,\delta)$, then $S\subseteq Q$ will be called \emph{reachable} if there exists a word
$w\in\Sigma^*$ such that $Qw=S$. We say
that $S$ is \emph{reducible} if there exists a word $w$ such that $|Sw|<|S|$, and we call $w$ a
\emph{reduction word} for $S$. Our algorithm is mainly based on the following observation:
\begin{property}\label{property:extension}
If a DFA $\mathcal{A}$ is synchronizing, and $\mathcal{B}$ is an extension of $\mathcal{A}$, then
$\mathcal {B}$ is synchronizing as well and its shortest synchronizing word is at most as long as the
shortest synchronzing word for $\mathcal{A}$.
\end{property}

The algorithm roughly runs as follows. We search for (super)critical DFAs on $n$ states, so a DFA is
discarded if it synchronizes faster, or if it does not synchronize at all. For a given DFA
$\mathcal{A} = (Q,\Sigma,\delta)$ which is not yet discarded or investigated, the algorithm does the
following:
\begin{enumerate}
\item If $\mathcal{A}$ is synchronizing and (super)critical, we have identified an example we are
searching for.
\item If $\mathcal{A}$ is synchronizing and subcritical, it is discarded, together with all its
possible extensions (justified by Property \ref{property:extension}).
\item If $\mathcal{A}$ is not synchronizing, then find an upper bound $L$ for how fast any synchronizing
extension of $\mathcal{A}$ will synchronize (see below). If $L < (n-1)^2$, then discard $\mathcal{A}$
and all its extensions. Otherwise, discard only $\mathcal{A}$ itself.
\end{enumerate}

The upper bound $L$ for how fast any synchronizing extension of $\mathcal{A}$ will synchronize, is
found by analyzing distances in the directed graph of the power automaton of $\mathcal{A}$. For
$S,T\subseteq Q$, the distance from $S$ to $T$ in this graph is equal to the length of the shortest
word $w$ for which $Sw=T$, if such a word exists.
We compute $L$ as follows:
\begin{enumerate}
\item Determine the size $|S|$ of a smallest reachable set. Let $m$ be the minimal distance from $Q$
to a set of size $|S|$.
\item For each $k\leq |S|$, partition the collection of irreducible sets of size $k$ into strongly
connected components. Let $m_k$ be the number of components plus the sum of their diameters.
\item For each reducible set of size $k\leq |S|$, find the length of its shortest
reduction word. Let $l_k$ be the maximum of these lengths.
\item Now note that a synchronizing extension of $\mathcal{A}$ will have a synchronizing
word of length at most
\[ L \; = \; m+\sum_{k=2}^{|S|}(m_k+l_k).  \]
\end{enumerate}

The algorithm performs a depth-first search. So after investigating a DFA, first all its extensions (not
yet considered) are investigated before moving on.
Still, we can choose which extension to pick first. We would like to choose
an extension that is likely to be discarded immediately together with all its extensions. Therefore,
we apply the following heuristic: for each possible extension $\mathcal{B}$ by one symbol, we count how
many pairs of states in $\mathcal{B}$ would be reducible. The extension for which this is maximal is
investigated first. The motivation is that a DFA is synchronizing if and only if each pair is reducible \cite{C64}.

Finally, we note that we have described a primitive version of the algorithm here. The algorithm
which has actually been used also takes symmetries into account, making it almost $n!$ times faster. \versie{}{For the source code, we refer to \cite{BDZ17}.}

\section{PFAs on at Most 6 States}
\label{secnfa}
In the remainder of this paper, we study PFAs and shortest carefully synchronizing word lengths. In this section, we focus on PFAs on at most 6 states. In
the next section, we construct PFAs with shortest carefully synchronizing words of exponential lengths for general $n$.

To find PFAs with small number of states and long shortest carefully synchronizing word, we exploit that
Property \ref{property:extension} also holds for PFAs. However, for PFAs it is not true that reducibility
of all pairs of states guarantees careful synchronization. Therefore, we apply a different search
algorithm.
\versie{In fact, we just choose the symbols of a long shortest synchronizing word from left to right.
More precisely, the symbols which are on the stack of the search function are always a prefix
of a possible synchronizing word. The search is pruned in the following three cases,
where $w$ is the prefix on the stack:
\begin{enumerate}

\item There exists a word $u$ consisting of the letters of $w$,
with $|u| < |w|$, such that either $Qu = Qw$, or $Qu$ and $Qw$ are both singletons;

\item The automaton $\mathcal{A}$, whose symbols are the letters
of $w$, has a synchronizing word which is smaller than the word length where we are targeting on;

\item The value of the upper bound $L$ for the automaton $\mathcal{A}$
is smaller than the word length where we are targeting on.

\end{enumerate}
If we just choose any possible symbol for each subsequent letter, the algorithm gives too many solution.
To reduce the number of solutions and speed up the algorithm even further, we only select symbols as follows,
where $w$ is the prefix on the stack and $a$ is a candidate new symbol:
\begin{enumerate}

\item If $Qwa = Qwb$ for a letter $b$ of $w$, then $a$ must be the first such letter $b$ (otherwise we prune
the search);

\item If $Qwa = Qwb$ does not hold for any letter $b$ of $w$, then $a$ must be undefined outside
$Qw$ (otherwise we prune the search).

\end{enumerate}
Note that we minimize the effect to $\mathcal{A}$ of adding symbol $a$ to it,
where we only may choose another symbol $a$ if the subset $Qwa$ will not be affected.
In 1., $\mathcal{A}$ even stays the same after adding symbol $a$. In 2., symbol $a$ is only
defined for states where it is needed, which restricts the applicability of $a$.

Due to 2. above, the algorithm did not find a solution of length $37$ with only
$6$ symbols for $n = 6$. But postprocessing all solutions for $n = 6$ did reveal
a solution of length $37$ with only $6$ symbols indeed.}{Technical details are discussed
in \cite{BDZ17}.}

For $n\leq 6$, our algorithm has identified the maximal
length of a shortest carefully synchronizing word in a PFA on $n$ states. The results are:
\[
\renewcommand{\arraystretch}{1.4}
\begin{array}{|c|c|c|c|c|c|}
\hline
\quad n\quad&\quad 2\quad&\quad3\quad&\quad4\quad&\quad5\quad&\quad6\quad\\
\hline
\quad\text{maximal\ length}\quad&1&4&10&21&37 \\
\hline
\end{array}
\vspace{3pt}
\]
We observe that PFAs exist for $n=4,5,6$ with shortest carefully synchronizing word lengths exceeding $(n-1)^2$.
Note that for $n=5,6$ this even exceeds the  Pin-Frankl bound $\frac{1}{6}(n^3-n)$ for DFAs from \cite{pin}.
Where for $n\geq 5$ no critical DFAs are known with more than three symbols, PFAs with long shortest carefully
synchronizing word lengths tend to have more symbols: for $n=4,5,6$ states the minimal numbers of symbols
achieving the maximal shortest carefully synchronizing word lengths 10, 21 and 37 are 3, 6, 6,
respectively. Below we give examples of PFAs on 4, 5 and 6 states reaching these lengths.
\begin{center}
\begin{tikzpicture}
\useasboundingbox (0,-0.5) rectangle (3,2.5);
\node[circle,draw,inner sep=0pt,minimum width=3mm] (1) at (0,2) {};
\node[circle,draw,inner sep=0pt,minimum width=3mm] (2) at (2,2) {};
\node[circle,draw,inner sep=0pt,minimum width=3mm] (3) at (2,0) {};
\node[circle,draw,inner sep=0pt,minimum width=3mm] (4) at (0,0) {};
\draw[->] (1) -- node[above,inner sep=1pt] {$a,c$} (2);
\draw[->] (2) -- node[right,inner sep=2pt] {$b$} (3);
\draw[->] (3) -- node[below,inner sep=3pt] {$b,c$} (4);
\draw[->] (4) -- node[left,inner sep=2pt] {$b,c$} (1);
\draw[->] (2) edge[out=90,in=0,looseness=10] node[right,inner sep=4pt] {$a$} (2);
\draw[->] (3) edge[out=-90,in=0,looseness=10] node[right,inner sep=4pt] {$a$} (3);
\draw[->] (4) edge[out=-90,in=-180,looseness=10] node[left,inner sep=4pt] {$a$} (4);
\end{tikzpicture}
\begin{tikzpicture}
\node[circle,draw,inner sep=0pt,minimum width=3mm] (0) at (0,1) {};
\node[circle,draw,inner sep=0pt,minimum width=3mm] (1) at (1.732,2) {};
\node[circle,draw,inner sep=0pt,minimum width=3mm] (2) at (1.732,0) {};
\node[circle,draw,inner sep=0pt,minimum width=3mm] (3) at (3.464,1) {};
\node[circle,draw,inner sep=0pt,minimum width=3mm] (4) at (5.196,1) {};
\draw[->] (0) edge[out=15,in=-135] node[pos=0.55,below,inner sep=3pt] {$b$} (1);
\draw[->] (1) edge[out=-165,in=45] node[pos=0.65,above,inner sep=9pt] {$a,d,e,f$} (0);
\draw[->] (1) -- node[right,inner sep=2pt] {$c$} (2);
\draw[->] (2) -- node[pos=0.55,below,inner sep=4pt] {$c$} (0);
\draw[->] (3) -- node[pos=0.45,above,inner sep=3pt] {$c$} (1);
\draw[->] (2) edge[out=45,in=-165] node[pos=0.45,above,inner sep=3pt] {$d$} (3);
\draw[->] (3) edge[out=-135,in=15] node[pos=0.45,below,inner sep=4pt] {$e$} (2);
\draw[->] (4) edge[out=165,in=15] node[above,inner sep=2pt] {$e$} (3);
\draw[->] (3) edge[out=-15,in=-165] node[below,inner sep=3pt] {$f$} (4);
\draw[->] (0) edge[out=225,in=135,looseness=10] node[left,inner sep=2pt] {$a$} (0);
\draw[->] (2) edge[out=225,in=315,looseness=10] node[pos=0.2,left,inner sep=3pt] {$a,b$} (2);
\draw[->] (3) edge[out=135,in=45,looseness=10] node[above,inner sep=2pt] {$a,b$} (3);
\draw[->] (4) edge[out=315,in=45,looseness=10] node[right,inner sep=3pt] {$a,b,c,d$} (4);
\end{tikzpicture}
\end{center}
The left one has two synchronizing words of length 10:
$abcabab\bcdot(b+c)ca$. The right one has unique shortest synchronizing word $\mathit{abcabdbebcabdbfbcdeca}$ of length 21.
\begin{center}
\begin{tikzpicture}
\node[circle,draw,inner sep=0pt,minimum width=3mm] (0) at (0,1) {};
\node[circle,draw,inner sep=0pt,minimum width=3mm] (1) at (2.464,1) {};
\node[circle,draw,inner sep=0pt,minimum width=3mm] (2) at (3.464,2) {};
\node[circle,draw,inner sep=0pt,minimum width=3mm] (3) at (3.464,0) {};
\node[circle,draw,inner sep=0pt,minimum width=3mm] (4) at (5.196,1) {};
\node[circle,draw,inner sep=0pt,minimum width=3mm] (5) at (6.928,1) {};
\draw[->] (0) edge[out=45,in=-180] node[above]{$b$} (2);
\draw[->] (1) -- node[pos=0.40,above,inner sep=2pt]{$a,b,d,e,f$} (0);
\draw[->] (1) -- node[pos=0.40, below,inner sep=4pt]{$c$} (3);
\draw[->] (2) -- node[pos=0.60,above,inner sep=4pt]{$b$} (1);
\draw[->] (3) edge[out=-180,in=-45] node[below]{$c$} (0);
\draw[->] (3) edge[out=15,in=-135] node[pos=0.55,below,inner sep=4pt]{$d$}(4);
\draw[->] (4) -- node[pos=0.45,above,inner sep=3pt]{$c$}(2);
\draw[->] (4) edge[out=-165,in=45] node[pos=0.55,above,inner sep=4pt]{$e$} (3);
\draw[->] (4) edge[out=-15,in=-165] node[below,inner sep=3pt] {$f$} (5);
\draw[->] (5) edge[out=165,in=15] node[above,inner sep=2pt] {$e$}(4);
\draw[->] (0) edge[out=225,in=135,looseness=10] node[left,inner sep=2pt] {$a$} (0);
\draw[->] (2) edge[out=135,in=45,looseness=10] node[pos=0.2,left,inner sep=2pt] {$a$} (2);
\draw[->] (3) edge[out=225,in=315,looseness=10] node[pos=0.2,left,inner sep=3pt] {$a,b$} (3);
\draw[->] (4) edge[out=135,in=45,looseness=10] node[above,inner sep=2pt] {$a,b$} (4);
\draw[->] (5) edge[out=315,in=45,looseness=10] node[right,inner sep=3pt] {$a,b,c,d$} (5);
\end{tikzpicture}
\end{center}
The shortest synchronizing word is  $ab^2ab^2cb^2ab^2db^2eb^2cb^2ab^2db^2fb^2cdecb^2a$ for this PFA on 6 states. It is unique and has length 37.

\section{Exponential Bounds for PFAs}
\label{secexp}
In this section, we construct for any $k \geq 3$ a strongly connected PFA on $n = 3k$ states and three symbols,
for which we show that it is
carefully synchronizing, and the shortest carefully synchronizing word has length $\Omega(\phi^{n/3})$
for $\phi = \frac{1+ \sqrt{5}}{2} = 1.618\cdots$. The set of states is
$Q = \{A_i, B_i, C_i \mid i = 1,\ldots,k\}$.
If a set $S \subseteq Q$ contains exactly one element of $\{A_i, B_i, C_i\}$ for every $i$,
it can be represented by a string over $\{A,B,C\}$ of length $k$. The idea of our
construction is that the PFA will mimic rewriting the string $C^2 A^{k-2}$ to the string
$C^2 A^{k-3} B$ with respect to the rewrite system $R$, which consists of the following three rules
\[ BBA \to AAB, \; CBA \to CAB, \; CCA \to CCB.\]
The key argument is that this rewriting is possible, but requires an exponential number of
steps.
This is elaborated in the following lemma, in which we use $\to_R$ for rewriting with respect to
$R$, that is, $u \to_R v$, if and only if $u = u_1 \ell u_2$ and $v = u_1 r u_2$, for strings
$u_1,u_2$ and a rule $\ell \to r$ in $R$. Its transitive closure is denoted by $\to_R^+$.
We write $\fib$ for the standard fibonacci function, defined by
$\fib(i) = i$ for $i=0,1$, and $\fib(i) = \fib(i-1) + \fib(i-2)$ for $i > 1$. It is well-known
that $\fib(n) = \Theta(\phi^n)$.

\begin{lemma}
\label{lemrewrlen}
For $k \geq 3$, we have $CCA^{k-2} \to_R^+ CCA^{k-3}B$. Furthermore,
the smallest possible number of steps for rewriting $CCA^{k-2}$ to a string ending in $B$, is
exactly $ \fib(k)-1$.
\end{lemma}

\begin{proof}
For the first claim we do induction on $k$. For $k=3$, we have $CCA \to_R CCB$.
For $k=4$, we have $CCAA \to_R CCBA \to_R CCAB$. For $k>4$, applying the induction
hypothesis twice, we obtain
\[CCA^{k-2} \to_R^+ CCA^{k-4}BA \to_R^+ CCA^{k-5}BBA \to_R CCA^{k-3}B.\]
For the second claim, we define the \emph{weight} $W(u)$ of a string $u = u_1 u_2 \cdots u_k$ over
$\{A,B,C\}$ of length $k$ by
$$
W(u) = \sum_{i : u_i = B} (\fib(i)-1).
$$
So every $B$ on position $i$ in
$u$ contributes $\fib(i)-1$ to the weight, and the other symbols have no weight.

Now we claim that
$W(v) = W(u)+1$
for all strings $u,v$ with $u \to_R v$ and $u,v$ only having $C$'s in the first two positions.
Since the $C$s only occur at positions 1 and 2, by applying $CCA \to
CCB$, the weight increases by $\fib(3)-1 = 1$ by the creation of $B$ on position 3,
and by applying $CBA \to CAB$, it increases by $\fib(4)-1 -(\fib(3)-1) = 1$  since $B$ on position 3
is replaced by $B$ on position 4.
By applying $BBA \to AAB$, the contributions to the weight $\fib(i)-1$ and $\fib(i+1)-1$ of the two
$B$s are replaced by $\fib(i+2)-1$ of the new $B$, which is an increase by 1 according to
the definition of $\fib$.

So this weight increases by exactly 1 at every rewrite step, hence it
requires exactly $\fib(k)-1$ steps, to go from the initial string  $CCA^{k-2}$ of weight 0 to the
weight $\fib(k)-1$ of a $B$ symbol on the last position $k$, if that is the only $B$, and more
steps if there are more $B$s.
\qed
\end{proof}

Now we are ready to define the PFA on $Q = \{A_i, B_i, C_i \mid i = 1,\ldots,k\}$ and three
symbols. The three symbols are a start symbol $s$, a rewrite symbol $r$ and a cyclic
shift symbol $c$. The transitions are defined as follows (writing $\bot$ for undefined):
\[
\renewcommand{\arraystretch}{1.4}
\begin{array}{|@{\quad}r@{\,=\,}l@{\qquad}r@{\,=\,}l@{\qquad}r@{\,=\,}l@{\qquad}l@{\quad}|}
\hline
\multicolumn{3}{|@{\quad}r@{\,=\,}}{A_is = B_is} &
\multicolumn{3}{l}{C_i s = C_i,} & \mbox{for $i = 1,2$}, \\[-5pt]
\multicolumn{3}{|@{\quad}r@{\,=\,}}{A_is = B_is} &
\multicolumn{3}{l}{C_i s = A_i,} & \mbox{for $i = 3,\ldots,k$}, \\
\hline
A_1r & \bot, & B_1r & A_1, & C_1r & C_1, & \\[-5pt]
A_2r & \bot, & B_2r & A_2, & C_2r & C_2, & \\[-5pt]
A_3r & B_3, & B_3r & \bot, & C_3r & B_2, & \\[-5pt]
A_ir & A_i, & B_ir & B_i, & C_ir & C_i, & \mbox{for $i = 4,\ldots,k$}, \\
\hline
A_ic & A_{i+1}, & B_ic & B_{i+1}, & C_ic & C_{i+1}, & \mbox{for $i = 1,\ldots,k-1$}, \\[-5pt]
A_kc & A_1, & B_kc & B_1, & C_kc & C_1. & \\
\hline
\end{array}
\vspace{3pt}
\]
A shortest carefully synchronizing word starts by $s$, since $r$ is not defined on all states and
$c$ permutes all states. After $s$, the set of reached states is
$S(CCA^{k-2}) = \{C_1,C_2,A_3,\ldots,A_k\}$. Here, for a string $u = a_1 a_2 \cdots a_k$
of length $k$ over $\{A,B,C\}$, we write $S(u)$ for the set of $k$ states, containing $A_i$
if and only if $a_i = A$, containing $B_i$ if and only if $a_i = B$, and containing $C_i$
if and only if $a_i = C$, for $i = 1,2,\ldots,k$. Note that for $x \in \{A,B,C\}$ and
$v \in \{A,B,C\}^{k-1}$, we have $S(vx)c = S(xv)$, so $c$ performs a cyclic shift on strings
of length $k$.

The next lemma states that the symbol $r$ indeed mimicks rewriting: applied on sets of the shape
$S(u)$, up to cyclic shift it acts as rewriting on $u$ with respect to $R$ defined above.

\begin{lemma}
\label{lemrewr}
Let $u$ be a string of the shape $CCw$, where $w \in \{A,B\}^{k-2}$.
If $u \to_R v$ for a string $v$, then $S(u)c^irc^{k-i} = S(v)$ for some $i<k$.

Conversely, if $u$ does not end in $B$ and there exists an $i$ such that
$r$ is defined on $S(u)c^i$, then $u \to_R v$ for a string $v$ of the shape
$CCw$, where $w \in \{A,B\}^{k-2}$.
\end{lemma}

\begin{proof}
First assume that $u \to_R v$.
If $u = u_1 BBA u_2$ and $v = u_1 AAB u_2$, then let $i = |u_2| + 3$, so
\begin{align*}
S(u) c^i r c^{k-i}
&= S(u_1 BBA u_2) c^i r c^{k-i} = S(BBA u_2 u_1) r c^{k-i} \\
&= S(AAB u_2 u_1) c^{k-i} = S(u_1 AAB u_2) = S(v).
\end{align*}
If $u = u_1 CBA u_2$ and $v = u_1 CAB u_2$, then again let $i = |u_2| + 3$, so
\begin{align*}
S(u) c^i r c^{k-i}
&= S(u_1 CBA u_2) c^i r c^{k-i} = S(CBA u_2 u_1) r c^{k-i} \\
&= S(CAB u_2 u_1) c^{k-i} = S(u_1 CAB u_2) = S(v).
\end{align*}
Finally, if $u = u_1 CCA u_2$ and $v = u_1 CCB u_2$, then $u_1 = \epsilon$ and the result follows
for $i=0$.

Conversely, suppose that $S(u)c^ir$ is defined. Since $S(u)c^k = S(u)$, we may assume
that $i < k$ and can write $u = u_1 u_2$, such that $|u_2| = i$.
Then $S(u)c^i = S(w)$, where $w = u_2 u_1$.
Write $w = a_1 a_2 \cdots a_k$. Since $S(u_2 u_1)r$ is defined, we
get $a_1 \neq A$,  $a_2 \neq A$ and $a_3 \neq B$.
Among these 8 cases, $a_1 = a_2 = a_3 = C$ does not occur since $u$ only contains 2 $C$s,
and $a_1 a_2 = BC$ or $a_2 a_3 = BC$ does not occur since $u$ does not end in $B$.
The remaining 3 cases are
$$
a_1a_2a_3 = BBA, \qquad a_1a_2a_3 = CBA, \qquad \mbox{and} \qquad a_1a_2a_3 = CCA,
$$
where $a_1a_2a_3$ is replaced by the corresponding right hand side of the rule by the
action of $r$. Then in $S(u)c^irc^{k-i}$, the two $C$s are on positions 1 and 2 again,
and we obtain $S(u)c^ir c^{k-i} = S(v)$ for a string $v$ of the given shape,
satisfying $u \to_R v$.
\qed
\end{proof}

Combining Lemmas \ref{lemrewrlen} and \ref{lemrewr} and the fact that $\fib(n) = \Omega(\phi^n)$,
we obtain the following.

\begin{corollary}
\label{correwr}
There is a word $w$ such that $S(CCA^{k-2})w = S(CCA^{k-3}B)$; the shortest word $w$ for which
$S(CCA^{k-2})w$ is of the shape $S(u) c^i$ for $u$ ending in $B$ has length $\Omega(\phi^k)$.
\end{corollary}

Now we are ready to prove the lower bound:

\begin{lemma}
\label{lemlb}
If $w$ is carefully synchronizing, then $|w| = \Omega(\phi^k)$.
\end{lemma}

\begin{proof}
Assume that $w$ is a shortest carefully synchronizing word. Then we already observed that the first
symbol of $w$ is $s$, and $w$ yields $S(CCA^{k-2})$ after the first step in the power automaton.
By applying only $c$-steps and $r$-steps, according to Lemma \ref{lemrewr},
only sets of the shape $S(u)c^i$ for which $CCA^{k-2} \to_R^+ u$  can be reached, until u ends in $B$.
In this process, each $r$-step corresponds to a rewrite step. Applying the
third symbol $s$ does not make sense, since then we go back to $S(CCA^{k-2})$. According
to Corollary \ref{correwr},
in the power automaton at least $\Omega(\phi^k)$ steps are required to reach a set
which is not of the shape $S(u) c^i$. So for reaching a singleton, the total number of steps is
at least $\Omega(\phi^k)$.
\qed
\end{proof}

Note that for the reasoning until now, the definition of $C_3 r = B_2$ did not play a role, and by
$s,r$ all states were replaced by states having the same index. But after the last symbol of $u$ has
become $B$, this $C_3 r = B_2$ will be applied, leading to a subset in which no state of the group
$A_3, B_3, C_3$ occurs any more.
\versie{We could have chosen $C_3 r = A_2$ or $C_3 r = C_2$ as well: it is just that $C_3 r = B_2$
makes $r$ injective, just like $q$.}{}
Now we arrive at the main theorem.

\begin{theorem}
For every $n$ there is a carefully synchronizing PFA on $n$ states and three symbols with shortest carefully synchronizing word length $\Omega(\phi^{n/3})$.
\end{theorem}

\begin{proof}
If $n = 3k$ we take our automaton, otherwise we add one or two states on which $r,c$ are undefined and $s$ maps to $A_1$,
having no influence on the argument. The bound was proved in Lemma \ref{lemlb}; it remains to
prove that the automaton is carefully synchronizing, that is, it is possible to end up in a
singleton in the power automaton.

Let $w$ be the word from Corollary \ref{correwr}.
Since $S(CCA^{k-2})w = S(CCA^{k-3}B)$
and the number of $c$'s in $w$ is divisible by $k$, we have $C_1 w = C_1$, $C_2 w = C_2$, $A_3 w
= A_3, \ldots, A_{k-1} w = A_{k-1}$, $A_k w = B_k$. Hence
\begin{alignat*}{4}
\{A_1,B_1,C_1\}swcr &=\,& \{C_1\}cr &=\,& \{C_2\} &\subseteq\,& \{A_1,B_1,C_1\}&c,  \\
\{A_2,B_2,C_2\}swcr &=& \{C_2\}cr &=& \{B_2\} &\subseteq &\{A_2,B_2,C_2\}&, \\
\{A_i,B_i,C_i\}swcr &=& \{A_i\}cr &=& \{A_{i+1}\} &\subseteq &\{A_i,B_i,C_i\}&c,
\quad \mbox{for $i=3,4,\ldots,k-1$,} \\
\{A_k,B_k,C_k\}swcr &=& \{B_k\}cr &=& \{A_1\} &\subseteq &\{A_k,B_k,C_k\}&c.
\end{alignat*}
So for all $i \ne 2$, $\{A_i,B_i,C_i\}swcr$ is contained in the cyclic successor $\{A_i,B_i,C_i\}c$
of $\{A_i,B_i,C_i\}$. $\{A_2,B_2,C_2\}swcr$ is just contained in $\{A_2,B_2,C_2\}$ itself.
Since for any $i$, one can take the cyclic successor of $\{A_i,B_i,C_i\}$ at most $k-1$ times before
ending up in $\{A_2,B_2,C_2\}$, we deduce that
$$
\{A_i,B_i,C_i\}(swcr)^{k-1} \subseteq \{A_2,B_2,C_2\}
\quad \mbox{for $i=1,2,\ldots,k$}.
$$
As $\{A_2,B_2,C_2\}s = \{C_2\}$, we obtain the carefully synchronizing word $(s w c r)^{k-1}s$
of the PFA.
\qed
\end{proof}

\versie{The word $(s w c r)^{k-1}s$ is a lot longer than necessary. In fact, one can
prove that only $\Omicron(k^2)$ $c$-steps and $\Omicron(k)$ $r$-steps and $s$-steps suffice
after $swcr$.}{}

\section{Reduction to Two Symbols}
\label{sectwos}
In this section we construct PFAs with two symbols and exponential shortest carefully synchronizing word length. We do this
by a general transformation to two-symbol PFAs, as was done before, e.g. in \cite{V16}. There a PFA on
$n$ states and $m$ symbols was transformed to a PFA on $mn$ states and two symbols, preserving
synchronization length. In the next theorem,
we improve this resulting number of states to $(m-1)n$ or even less, only needing a mild extra condition.
Using this result, we reduce our 3-symbol PFA with synchronizing length $\Omega(\phi^{n/3})$ to a
2-symbol PFA with synchronizing length $\Omega(\phi^{n/5})$.

\begin{theorem}
\label{lem2sym}
Let $P = (Q, \Sigma)$ be a carefully synchronizing PFA with $|Q| = n$, $|\Sigma| = m$, and
shortest carefully synchronizing word length $f(n)$. Assume $s \in \Sigma$ and $Q' \subseteq Q$
satisfy the following properties.
\begin{enumerate}
\item there is some number $p$ such that all symbols are defined on $Q s^p$ for a complete symbol $s$,
\item $qs = q$ for all $q \in Q'$, and
\item $qa = qb$ for all $q \in Q'$ and all $a,b \in \Sigma \setminus \{s\}$.
\end{enumerate}
Let $n' = n - |Q'|$. Then there exists a carefully
synchronizing PFA on $n + n'\bcdot(m-2)$ states and 2 symbols, with shortest carefully synchronizing word
length at least $f(n)$.
\end{theorem}

Note that if $Q' = \emptyset$ then only requirement 1 remains, and the resulting number of states is
$n + n'\bcdot(m-2) = (m-1)n$.

\begin{proof}
Write $Q = \{1,2,\ldots,n\}$, $Q' = \{n'+1,\ldots,n\}$,
and $\Sigma = \{s,a_1,\ldots,a_{m-1}\}$.
Let the states of the new PFA be
$P_{1,j}$ for $j = 1,\ldots,n$ and
$P_{i,j}$ for $i = 2,\ldots,m-1$, $j = 1,\ldots,n'$. Define the following two symbols $a,b$ on
these states:
\[P_{i,j} a = \begin{cases}
P_{i+1,j}, & \text{if $i < m-1, j \leq n'$}, \\
P_{1,j s}, & \text{if $i = m-1, j \leq n'$}, \\
P_{1,j}, & \text{if $i = 1, j > n'$}.
\end{cases}
\qquad
\begin{array}{cccccc}
P_{1,1} & \cdots & P_{1,n'} & P_{1,n'+1} & \cdots & P_{1,n} \\
P_{2,1} & \cdots & P_{2,n'} && \\
\vdots & & \vdots &&\\
P_{m-1,1} & \cdots & P_{m-1,n'} &&
\end{array}
\]
and $P_{i,j} b = P_{1,j a_i}$, for all $i = 1,\ldots,m-1$ and $j = 1,\ldots,n$
for which $P_{i,j}$ exists and $j a_i$ is defined.

If we arrange the states as indicated above, then on the leftmost $n'$ columns,
$a$ moves the states one step downward if possible, and
for the bottom row jumps to the top row and
acts there as $s$. For the remainder of the top row $a$ also acts as $s$
(which is the identity). On the leftmost $n'$ columns, the symbol $b$ acts
as $a_i$ on row $i$ and then jumps to the top line. For the remainder of the
top row, all $a_i$ act in the same way and $b$ acts likewise.

Define
$\psi(a_i) = a^{i-1} b$ for $i = 1,\ldots,m-1$, and $\psi(s) = a^{m-1}$. Then on the top line
$\psi(a_i)$ acts in the same way as $a_i$ in the original PFA. Similarly, $\psi(s)$ acts as $s$.
On any other row, $\psi(s)$ acts as $s$, too. Since every symbol $a_i$ is
defined on $qs^p$ for every $q\in Q$, we obtain that $\psi(s)^pb=a^{(m-1)p}b$
is defined on every state and ends up in the top row.

Assume that $w$ is carefully synchronizing in the original PFA.
Then by the above observations, $a^{(m-1)p} b \psi(w)$ is carefully synchronizing in the new PFA.
Conversely, any carefully synchronizing word of the new PFA can be written as $\psi(w)a^j$,
where $0\leq j\leq m-2$ and $\psi(w)$ is a
concatenation of blocks of the form $\psi(l),l\in\Sigma$. Now note that
$a^j$ can never synchronize two distinct states in the top row. Therefore,
$\psi(w)$ synchronizes the top row and consequently $w$ is synchronizing in
the original PFA. Clearly $|\psi(w)a^j|\geq|w|\geq f(n)$.
\qed
\end{proof}

We apply Theorem \ref{lem2sym} to our basic construction with $3k$ states and $m=3$ symbols;
note that $s,c$ are defined on all states and $r$ is defined on $Qs$, so the requirements of
Theorem \ref{lem2sym} hold for $p=1$. As $r$ and $c$ act differently on all states, the only option for
$Q'$ is $Q' = \emptyset$.
Hence we obtain a carefully synchronizing PFA on $(m-1)3k = 6k$ states and two symbols, with
shortest carefully synchronizing word length $\Omega(\phi^k)$.
For $n$ being the number of states of the new PFA, this is $\Omega(\phi^{n/6})$.

However, instead of our three symbols $s,c,r$ we also get careful synchronization on the three
symbols $s,c,rc$ with careful synchronization length of the same order. But then for $i = 4,\ldots,k$ we
have $A_i s = A_i$ and $A_i c = A_i rc$, so we may choose $Q' = \{A_4,\ldots,A_k\}$ in Theorem
\ref{lem2sym}, by which $n' = 3k - (k-3) = 2k+3$, yielding a PFA on two symbols and $5k+3$ states.
This results in the following theorem, where for $n$ not of the shape $5k+3$ we add $\leq 4$ extra states
to achieve this shape, where $b$ is undefined on the new states and $a$ maps the new states to existing states.

\begin{theorem}
For every $n$ there is a carefully synchronizing PFA on $n$ states and two symbols with shortest carefully synchronizing word length $\Omega(\phi^{n/5})$.
\end{theorem}

\section{Further Optimizations}\label{secopt}
Some further optimizations are possible.  For instance, for any $h \geq 2$ we can take $h+1$
rewrite rules
$$
C^i B^{h-i} A \to C^i A^{h-i} B
$$
for $i = 0,\ldots,h$, and construct a
PFA on the $n = 3k$ states $A_i,B_i,C_i$ for $i = 1,\ldots,k$ with a similar
$s$, $c$, and $r$, mimicking the
rewrite rules in which the rewriting takes place in the states with indexes $\leq h+1$.
For $h=2$, this coincides with our construction, but for $h > 2$, this gives a better bound
$\Omega(a^k)$, where $a$ is the real zero of $x^h - x^{h-1} - \cdots - x - 1$ in
between $3/2$ and 2. As this value tends to 2 for increasing $h$, for every $\epsilon > 0$, we
achieve the bound $\Omega((2 - \epsilon)^{n/3})$ for three symbols and
$\Omega((2 - \epsilon)^{n/5})$ for two symbols.
\versie{

We can also add additional letters in between $A$ and $B$ to the rewrite system, say
$m - 2$ letters $X = X^{(1)}, X^{(2)}, \ldots, X^{(m-2)}$, and take $(h+1)(m-1)$
rewrite rules
\begin{align*}
C^i B^{h-i} A &\to C^i A^{h-i} X^{(1)} \\
C^i B^{h-i} X^{(j)} &\to C^i A^{h-i} X^{(j+1)} \\
C^i B^{h-i} X^{(m-2)} &\to C^i A^{h-i} B
\end{align*}
for $i = 0,\ldots,h$ and $j=1,\ldots,m-3$. Using these rewrite rules, we can construct
a PFA on $n = (m+1)k$ states $A_i,X_i,B_i,C_i$ for $i = 1,\ldots,k$, where $X_i$
is a sequence of $m-2$ states. We describe the symbols $s$, $c$, and $r$ by their actions
on some $(m+1)$-tuples, e.g.\@
$(A_i,X_i,B_i,C_i) \, s = (A_i s,X_i^{(1)} s,\ldots,X_i^{(m-2)} s,B_i s,C_i s)$
if $X_i = X_i^{(1)}, X_i^{(2)}, \ldots, X_i^{(m-2)}$:
\begin{align*}
(A_i,X_i,B_i,C_i) \, s &= \begin{cases}
(C_i,\ldots,C_i,C_i), & \text{if $i = 1,\ldots,h$} \\
(A_i,A_i,\ldots,A_i), & \text{otherwise}
\end{cases} \\
(A_i,X_i,B_i,C_i) \, c &= \begin{cases}
(A_1,X_1,B_1,C_1), & \text{if $i = k$} \\
(A_{i+1},X_{i+1},B_{i+1},C_{i+1}), & \text{otherwise}
\end{cases} \\
(A_i,X_i,B_i,C_i) \, r &= \begin{cases}
(\bot,\ldots,\bot,A_i,C_i), & \text{if $i = 1,\ldots,h$} \\
(X_i,B_i,\bot,B_{i-1}), & \text{if $i = h+1$} \\
(A_i,X_i,B_i,C_i), & \text{otherwise}
\end{cases}
\end{align*}
In the proof of Lemma \ref{lemrewrlen}, we replace $\fib$ by a function $f$
for which
\begin{gather*}
f(1) = f(2) = \cdots = f(h) = \tfrac{1}{(m-1)h-1} \qquad \mbox{and} \\
f(i+h) = (m-1) \cdot \big( f(i) + f(i+1) + \cdots + f(i+h-1)\big)
\end{gather*}
for all $i \ge 1$. Next, we take
$$
W(u) = \sum_{j=1}^{m-2} j \cdot \sum_{i : u_i = X^{(j)}} \big(f(i) - \tfrac{1}{(m-1)h-1}\big)
+ (m-1) \cdot \sum_{i : u_i = B} \big(f(i) - \tfrac{1}{(m-1)h-1}\big).
$$
Just like in the proof of Lemma \ref{lemrewrlen}, the smallest possible number of steps
for rewriting $C^hA^{k-h}$ to a string ending in $B$ is exactly $f(k) - \tfrac{1}{(m-1)h-1}$.

To see that $\Omega((m-\epsilon)^k)$ is obtainable, take
$$
h \geq \frac{\log(m-1) - \log \epsilon}{\log (m-\epsilon)}
$$
We prove that $f(i) \geq \tfrac{1}{(m-1)h-1} (m-\epsilon)^{i-h}$ for all $i$.
It is satisfied if $1 \le i \le h$. By induction, we deduce that
\begin{align*}
f(i+h) &= (m-1) \cdot \big( f(i) + f(i+1) + \cdots + f(i+h-1)\big) \\
&\geq \tfrac{m-1}{(m-1)h-1} \cdot \big((m-\epsilon)^{i-h} + (m-\epsilon)^{i+1-h} + \cdots +
       (m-\epsilon)^{i-1}\big) \\
&= \tfrac{m-1}{(m-1)h-1} \cdot \frac{(m-\epsilon)^i - (m-\epsilon)^{i-h}}{(m-\epsilon) - 1} \\
&= \tfrac{1}{(m-1)h-1} (m-\epsilon)^i \cdot (m-1)
       \frac{1-(m-\epsilon)^{-h}}{(m-\epsilon) - 1} \\
&\geq \tfrac{1}{(m-1)h-1} (m-\epsilon)^i \cdot (m-1)
       \frac{1-\tfrac{\epsilon}{m-1}}{(m-\epsilon) - 1} \\
&= \tfrac{1}{(m-1)h-1} (m-\epsilon)^i
\end{align*}
so $f(i) \geq \tfrac{1}{(m-1)h-1} (m-\epsilon)^{i-h}$ for all $i$ indeed.

Now take $m = 4$. Then $n = 5k$, and it takes
$$
f(k) - \tfrac{1}{(m-1)h-1} = \Omega((m-\epsilon)^k) =
\Omega((4 - \epsilon)^{n/5}) = \Omega((2 - \epsilon)^{2n/5})
$$
steps to obtain a string ending in $B$.

To improve the construction with 2 symbols, we replace the start symbol $s$
by a symbol $s'$, to obtain a bigger set $Q'$ in Theorem \ref{lem2sym}:
$$
(A_i,X_i,B_i,C_i) \, s' = \begin{cases}
(C_i,\ldots,C_i,C_i), & \text{if $i = 1,\ldots,h$} \\
(A_i,A_i,\ldots,A_i), & \text{if $i = k-3, k-2, k-1, k$} \\
(A_i,X_i,B_i,A_i), & \text{otherwise}
\end{cases}
$$
Indeed, we can take
$$
Q' = \{A_{h+2},A_{h+3},\ldots,A_k,X_{h+2},X_{h+3},\ldots,X_{k-4},
       B_{h+2},B_{h+3},\ldots,B_{k-4}\}
$$
so that $Q'$ contains $m (k-h-5) + 4$ states. Hence
$$
n' = n - |Q'| = (m+1)k - (m (k-h-5) + 4) = k + m(h+5) - 4 = k + O(1),
$$
and we obtain a PFA on $n + n' = (m+1)k + k + O(1) = (m+2)k + O(1)$
states and two symbols.

Since $s = s'(c^{k-1}s')^k$, we deduce that the PFA with symbols $s',c,rc$
is synchronizing. We will show below that the shortest carefully
synchronizing word length of the PFA with symbols $s',c,rc$ is
$\Omega((m-\epsilon)^k)$, just as for $s',c,rc$.
Now take $m = 4$. For $n$ being the number of states of the PFA with two symbols,
the length of the shortest synchronizing word is
$$
\Omega\big((m-\epsilon)^{(n - O(1))/(m+2)}\big)
= \Omega((4-\epsilon)^{n/6}) = \Omega((2 - \epsilon)^{n/3})
$$

So it remains to show that replacing $s$ by $s'$ does not change
the estimate $\Omega((m-\epsilon)^k)$ for the shortest carefully
synchronizing word length of the PFA. As opposed to $s$, we can abuse
$s'$ to get from $Q$ to $S(C^hB^{k-h-4}A^4)$ in polynomially many steps
in the power automaton, but we cannot cheat beyond that.
Hence it suffices to show that
$$
W(C^hA^{k-h-1}B) - W(C^hB^{k-h-4}A^4) \ge
\tfrac12 \big(W(C^hA^{k-h-1}B) - W(C^hA^{k-h})\big)
$$
which is equivalent to
$$
W(C^hB^{k-h-4}A^4) \le \tfrac12 W(C^hA^{k-h-1}B)
$$
Since $W(C^hA^{k-h-3}BA^2) \le
\tfrac12 W(C^hA^{k-h-1}B)$, we see that it suffices to show that
$W(C^hB^{k-h-4}A^4) \le W(C^hA^{k-h-3}BA^2)$, i.e.\@
\begin{equation} \label{Weq}
W(C^hB^iA^{k-h-i}) \le W(C^hA^{i+1}BA^{k-h-i-2})
\end{equation}
for $i = k-h-4$.

We prove \eqref{Weq} by induction on $i$.
It is clear that \eqref{Weq} holds for $i = 0$ and $i = 1$.
By induction, we deduce that
\begin{align*}
W(C^hB^{i+2}A^{k-h-i-2})
&= W(C^hB^{i}A^{k-h-i}) + W(C^hA^{i}B^2A^{k-h-i-2}) \\
&\le W(C^hB^{i}A^{k-h-i}) + W(C^hA^{i+2}BA^{k-h-i-3}) \\
&\le W(C^hA^{i+1}BA^{k-h-i-2}) + W(C^hA^{i+2}BA^{k-h-i-3}) \\
&= W(C^hA^{i+1}B^2A^{k-h-i-3}) \\
&\le W(C^hA^{i+3}BA^{k-h-i-4}),
\end{align*}
so \eqref{Weq} holds if $0 \le i \le k-h-2$.}{For further improvements to $\Omega((4 - \epsilon)^{n/5}) =
\Omega((2 - \epsilon)^{2n/5})$ for three symbols and $\Omega((4 - \epsilon)^{n/6}) =
\Omega((2 - \epsilon)^{n/3})$ for two symbols, we refer to the extended version
\cite{BDZ17}.}

\section{Conclusions}
\label{secconcl}
For every $n$ we constructed a PFA on $n$ states and 3 symbols for which careful synchronization is forced to mimic rewriting with respect to a string rewriting system. This system requires an exponential number of steps to reach a string of a particular shape. The resulting exponential synchronization length is much larger than the cubic upper bound for synchronization length of DFAs. We show that for $n=4$ the shortest synchronization length for a PFA already can exceed the maximal shortest synchronization length for a DFA. For $n=4,5,6$ we found greatest possible shortest synchronization lengths, both for DFAs and PFAs, where for DFAs until now this was only fully investigated for $n \leq 4$, that is, by not assuming any bound on the number of symbols. Both for DFAs and PFAs better techniques are needed to do the same analysis for $n=7$ or higher.

\bibliography{ref}

\end{document}